\documentclass[12pt,oneside]{amsart}
\usepackage{amsmath}
\usepackage {multirow}
\usepackage {lscape}

\usepackage {amsfonts}
\usepackage {amsmath}
\usepackage {amsthm}
\usepackage {amssymb}
\usepackage {framed}
\usepackage {amsxtra}
\usepackage {enumerate}
\usepackage {graphicx}
\usepackage{color}
\usepackage{graphicx}
\usepackage[left=0.9in, right=0.9in, top=1.15 in, footskip=0.5in, bottom=1.15 in]{geometry}
\makeatletter

\theoremstyle{definition}
\newtheorem{df}{Definition} [section]

\theoremstyle{plain}
\newtheorem{thm}[df]{Theorem}

\newtheorem{lemma}[df]{Lemma}

\newtheorem{claim}[df]{Claim}
\newtheorem{conj}[df]{Conjecture}

\newtheorem{fact}[df]{Fact}

\title{Outdegree conditions forcing short cycles in digraphs}

\author{Dan Ismailescu}
\address{Mathematics Department, Hofstra University, Hempstead, NY 11549, USA.}
\email{dan.p.ismailescu@hofstra.edu}

\author{Joonsoo Lee}
\address{Dwight-Englewood High School, 315 E Palisade Ave, Englewood, NJ 07631, USA.}
\email{jlee20@d-e.org}

\author{Andrew Yang}
\address{The Hotchkiss School, 11 Interlaken Rd, Lakeville, CT 06039, USA.}
\email{ayang21@hotchkiss.org}

\begin{document}

\begin{abstract}
Given a positive integer $m\ge 3$, let $ch(m)$ be the smallest positive constant with the following property:

\emph{ Every simple directed graph on $n\ge 3$ vertices all whose outdegrees are at least $ch(m)\cdot n$ contains a directed cycle of length at most $m$.}

Caccetta and H\"{a}ggkvist conjectured that $ch(m)=1/m$, which if true, would be the best possible.
In this paper, we prove the following result:

\emph{ For every integer $m\ge 3$, let $\alpha(m)$ be the unique real root in $(0,1)$ of the equation}
\begin{equation*}
(1-x)^{m-2}=\frac{3x}{2-x}.
\end{equation*}
Then $ch(m)\le \alpha(m)$.

This generalizes results of Shen who proved that $ch(3)\le 3-\sqrt{7}<0.35425$, and Liang and Xu who showed that $ch(4)< 0.28866$ and $ch(5)<0.24817$.

We then slightly improve the above inequality by using the minimum feedback arc set approach initiated by Chudnovsky, Seymour, and Sullivan. This results in extensions of the findings of Hamburger, Haxell and Kostochka (in the case $m=3$), and Liang and Xu (in the case $m=4$).
\end{abstract}

\maketitle
\thispagestyle{empty}

\section{\bf{Introduction}}
Let $G=(V,E)$ denote a digraph on $n$ vertices with no loops, no cycles of length 2, and no
multiple edges. Let $d^+_G(v)$, or simply $d^+(v)$ if $G$ is specified,
denote the outdegree of the vertex $v$ in $G$. If $G$ has at least one directed cycle,
the minimum length of a cycle in $G$ is called the \emph{girth} of $G$.

In 1978, Caccetta and H\"{a}ggkvist \cite{CH} proposed the following

\begin{conj}\label{CHconj}
Given a positive integer $m\ge 3$, let $ch(m)$ be the smallest positive constant with the following property: every simple directed graph on $n\ge 3$ vertices all whose outdegrees are at least $ch(m)\cdot n$ contains a directed cycle of length at most $m$. Then, $c(m)=1/m$.
\end{conj}

If true, the above result is optimal as shown by a construction of Behzad, Chartrand, and Wall \cite{BCW}.
In the case $m=3$ there are several types of extremal digraphs; we refer to \cite{sullivansurvey} for a thorough survey of the literature.

While the general conjecture is still open, some partial results have been obtained.

Caccetta and H\"{a}ggkvist \cite{CH} proved that $ch(3)\le (3-\sqrt{5})/2<0.38197$ by using an inductive argument. Bondy \cite{bondy} used a counting technique to improve this to $ch(3)\le (2\sqrt{6}-3)/5<0.3798$. Soon after, Shen \cite{shen1} showed that $ch(3)\le 3-\sqrt{7}<0.35425$, which was later further improved by Hamburger, Haxell, and Kostochka \cite{HHK} who proved that $ch(3)<0.35312$.

The best currently known bound is due to  Hladk\'{y}, Kr\'{a}l', and Norin \cite{HKN} who used Razborov's \cite{razborov} flag algebra approach
to prove that $ch(3)<0.3465$.

Liang and Xu \cite{LX2, LX1} considered the case $m=4$: they proved that $ch(4)<0.28866$, which they later improved to $ch(4)<0.28724$. They also showed in \cite{LX5} that $ch(5)<0.24817$.

Some general bounds are also known.

Chv{\'a}tal and  Szemer{\'e}di \cite{CS} proved that $ch(m)\le 2/m$.
Improving results in \cite{CS, nishimura}, Shen \cite{shen3} showed that
\begin{equation}\label{shenbounds}
ch(m)\le \frac{3\ln((2+\sqrt{7})/3)}{m-3}<\frac{1.3121}{m-3},\quad \text{and}\quad  ch(m) \le \frac{1}{m-73},
\end{equation}
for all $m\ge 4$, and all $m\ge 74$, respectively.
It follows that asymptotically, $ch(m)\thicksim 1/m$.

Shen \cite{shen2} also proved that for a given $m\ge 3$, the number of counterexamples to the Caccetta-H\"{a}ggkvist conjecture, if any, is finite.

In this paper we generalize a technique of Shen \cite{shen1}, and Liang and Xu \cite{LX2, LX5} in the following
\begin{thm}\label{mainthm1}
For every integer $m\ge 3$, let $\alpha:=\alpha(m)$ be the unique root in $(0,1)$ of the equation
\begin{equation*}
(1-x)^{m-2}=\frac{3x}{2-x}.
\end{equation*}
Then, any digraph on $n$ vertices with minimum outdegree at least $\alpha n$ contains a directed cycle of length at most $m$, that is, $ch(m)\le \alpha(m)$.
\end{thm}
In the next section we present a proof of this result. In section \ref{chudnovsky} we improve this bound by using a technique suggested by Chudnovsky, Seymour, and Sullivan \cite{CSS}. This extends findings of Haxell, Hamburger, and Kostochka \cite{HHK} in the case $m=3$, and Liang and Xu \cite{LX1} in the case $m=4$.

\section{\bf Proof of Theorem \ref{mainthm1}}

We prove Theorem \ref{mainthm1} by induction on $n$. It is easily seen that the statement is valid for $n=3$. Let us now suppose that the theorem  holds for all digraphs with fewer that $n$ vertices, and let $D=(V,E)$ be a counterexample with $n$ vertices. Thus, $D$ is a directed graph with minimum outdegree at least $\lceil \alpha n\rceil$ and $D$ contains no directed cycles of length at most $m$. Without loss of generality, we can assume that $D$ is $r$-outregular, where $r= \lceil \alpha n\rceil$, that is, every vertex is of outdegree $r$ in $D$. We intend to reach a contradiction.

We introduce some notation following \cite{shen1, LX2, LX5}.

\noindent For any vertex $v\in V(D)$ let
\begin{align*}
&N^+(v)=\{u\in V(D)\,|\, (v,u)\in E(D)\},\, |N^+(v)|=d^+(v)=r,\,\, \text{the outdegree of}\,\,v, \\
&N^-(v)=\{u\in V(D)\,|\, (u,v)\in E(D)\},\, |N^-(v)|=d^-(v),\,\,\,\,\, \text{the indegree of}\,\,v.
\end{align*}
We say that $(u,v,w)$ is a \emph{transitive triangle} if $(u,v), (u,w), (v,w) \in E(D)$. The edge $(u,v)$ is called the base of the transitive triangle.

For any edge $(u,v)\in E(D)$ define
\begin{align*}
&p(u,v):=|N^+(v)\setminus N^+(u)|,\,\text{the number of induced $2$-paths whose first edge is}\, (u,v),\\
&q(u,v):=|N^-(u)\setminus N^-(v)|,\,\text{the number of induced $2$-paths whose second edge is}\, (u,v),\\
&t(u,v):=|N^+(u)\cap N^+(v)|,\,\text{the number of transitive triangles whose base is}\, (u,v).
\end{align*}
Note that $p(u,v)+t(u,v)=|N^+(v)|=r$.

The following lemma was proved by Shen \cite{shen1} in the case $m=3$, and by Liang and Xu \cite{LX2,LX5} when $m=4, 5$.
\begin{lemma}\label{mainlemma1}
For every edge $(u,v)\in E(D)$ we have that
\begin{equation}\label{mainineq1}
n\ge \frac{1-(1-\alpha)^{m-2}}{\alpha}\,r+d^{-}(v)+q(u,v)+(1-\alpha)^{m-2}\,t(u,v).
\end{equation}
\end{lemma}
\begin{proof}
Note that the sets $N^+(v)$, $N^{-}(v)$, and $N^{-}(u)\setminus N^{-}(v)$ are mutually disjoint otherwise there exists a directed triangle.

We divide the proof into two cases depending on whether $t(u,v)=0$ or $t(u,v)>0$.

{\bf \emph{Case 1. $t(u,v)=0$}}

Since the sets mentioned above are disjoint, we have
\begin{align*}
&n\ge |N^+(v)|+|N^-(v)|+|N^-(u)\setminus N^-(v)|,\,\,\text{from which}\\
&n\ge r+d^-(v)+ q(u,v),
\end{align*}
and this proves (\ref{mainineq1}) in the case $m=3$ and $t(u,v)=0$.

One may suppose that $m\ge 4$. We need the following
\begin{claim}\label{claimtuv0}
There exists $m-3$ subsets of $V(D)$, denoted $S_1, S_2, \ldots, S_{m-3}$, such that for every $1\le k\le m-3$ the following properties hold
\begin{align*}
(a)\, &|S_1|+|S_2|+\cdots + |S_k| \ge (1-\alpha)r+(1-\alpha)^2r+\cdots +(1-\alpha)^k r.\\
(b)\, &\text{The distance between vertex}\,\, v\,\, \text{and any vertex of}\,\, S_k\, \text{is at most}\,\, k+1.\\
(c)\, &\text{The sets}\,\, N^+(v), N^{-}(v), N^{-}(u)\setminus N^{-}(v),\,\, S_1, S_2,\ldots S_{k-1} \,\,\text{and,}\,\, S_k\,\,\text{are mutually disjoint.}
\end{align*}
\end{claim}
\begin{proof}
We prove the above claim by finite induction on $k$.

For $k=1$, let $G_1$ be the subdigraph of $D$ induced by $N^{+}(v)$. Since $u \in N^-(v)$ and $N^-(v)\cap N^{+}(v)=\emptyset$, the order of this subgraph is strictly smaller than $n$. By the induction hypothesis, there exists a vertex $w_1 \in N^{+}(v)$ whose outdegree in $G_1$ is $\le \alpha|N^{+}(v)|=\alpha r$. Recall that $w_1$ has outdegree $r$ in $D$.

Let $S_1$ be the set of outneighbors of $w_1$ not contained in $N^{+}(v)$.

Then, clearly $|S_1|\ge r - \alpha r= (1-\alpha)r$. It also follows from the definition of $S_1$ that the distance between vertex $v$ and any vertex of $S_1$ is exactly $2$. In addition, we have $N^+(v)\cap S_1=\emptyset$.

If $x\in N^-(v)\cap S_1$ then $(v,w_1,x)$ is a directed triangle. If $x\in N^-(u)\cap S_1$ then $(u,v,w_1,x)$ is a directed $4$-cycle. This contradicts the assumption that $D$ does not contain directed cycles of length at most $m$. This proves the claim in the case $k=1$.

Suppose that for some $2\le k \le m-3$ we have found the sets $S_1, S_2, \ldots S_{k-1}$ with the properties stated in Claim \ref{claimtuv0}.

Let $G_{k}$ be the subdigraph of $D$ induced by $N^+(v) \sqcup S_1 \sqcup S_2\cdots \sqcup S_{k-1}$. Note that since $u\in N^-(v)$, statement $(c)$  implies that $u$ is not a vertex of this subgraph. By the minimality of $D$, there exists a vertex $w_{k} \in N^+(v) \sqcup S_1 \sqcup S_2\cdots \sqcup S_{k-1}$  whose outdegree within $G_k$ is no greater than $\alpha|N^+(v) \sqcup S_1 \sqcup S_2\cdots \sqcup S_{k-1}|= \alpha(|N^+(v)|+|S_1|+|S_2|+\cdots+ |S_{k-1}|)$.

On the other hand, $w_k$ is a vertex of $D$, so $d^+_D(w_k) = r = \lceil \alpha n \rceil$.

Let $S_{k} $ be the set of all outneighbors of $w_k$ which are not in $N^+(v) \sqcup S_1 \sqcup S_2\cdots \sqcup S_{k-1}$. Then,
\begin{align*}
&|S_k|\ge r-\alpha(|N^+(v)|+|S_1|+|S_2|+\cdots+ |S_{k-1}|)=r-\alpha r- \alpha(|S_1|+|S_2|+\cdots |S_{k-1}|)\implies\\
&|S_1|+|S_2|+\cdots |S_k|\ge (1-\alpha)r+(1-\alpha)(|S_1|+|S_2|+\cdots |S_{k-1}|)\implies\\
&|S_1|+|S_2|+\cdots |S_k|\ge (1-\alpha)r+(1-\alpha)^2r+\cdots (1-\alpha)^k r,
\end{align*}
and this proves the inductive step for statement $(a)$.

Since $w_k\in N^+(v) \sqcup S_1 \sqcup S_2 \sqcup \cdots \sqcup S_{k-1}$, the distance between $v$ and $w_k$ is no greater than one of the values $1, 2, \ldots k$, depending on whether $w_k\in N^+(v)$, or $w_k\in S_i$, for some $1\le i \le k-1$. Since $S_k$ is a subset of the outneighborhood of $w_k$, the inductive step for part $(b)$ of the claim follows.

For proving statement $(c)$ it suffices to show that $S_k$ is disjoint from both $N^-(v)$ and $N^-(u)$.
 If $S_k\cap N^-(v)\ne \emptyset$ then, since the distance from $v$ to any vertex of $S_k$ is no greater than $k+1$, one would obtain a directed cycle of length at most $k+2$. Similarly, if $S_k \cap N^-(u)\neq\emptyset$ then one obtains a directed cycle of length at most $k+3$. In both cases, we contradict the fact that $D$ contains no directed cycles of length at most $m$.
 This concludes the proof of Claim \ref{claimtuv0}
\end{proof}
We can now complete the proof of Case 1.

Since the sets $N^+(v), N^-(v), N^-(u)\setminus N^-(v), S_1, S_2, \ldots, \,\,\text{and}\,\, S_{m-3}$ are mutually disjoint
\begin{align*}
&n\ge |N^+(v)| + |N^-(v)|+ |N^-(u)\setminus N^-(v)|+ |S_1|+ |S_2|+ \cdots + |S_{m-3}| \implies\\
&n\ge r+ d^-(v)+q(u,v)+ (1-\alpha)r+(1-\alpha)^2r+\cdots +(1-\alpha)^{m-3}r \implies\\
&n\ge \frac{1-(1-\alpha)^{m-2}}{\alpha}\,r+d^{-}(v)+q(u,v),\quad \text{as claimed in}\, \eqref{mainineq1}.
\end{align*}
{\bf \emph{Case 2. $t(u,v)>0$}}

In this case there exists some vertex $w_1 \in N^+(u)\cap N^+(v)$ which has outdegree no greater than $\alpha|N^+(u)\cap N^+(v)|=\alpha t(u,v)$ in the subdigraph of $D$ induced by $N^+(u)\cap N^{+}(v)$. It follows that the outdegree of $w_1$ in the subdigraph induced by $N^+(v)$ cannot exceed $\alpha t(u,v)+p(u,v)= r-(1-\alpha) t(u,v)$.

Let $S_1$ be the set of outneighbors of $w_1$ not contained in $N^{+}(v)$.

Then, clearly $|S_1|\ge (1-\alpha)t(u,v)$. Also, it follows from the definition of $S_1$ that the distance between vertex $w_1$ and any vertex of $S_1$ is exactly $1$. Since $(u,w_1)$ and $(v,w_1)$ are directed edges in $D$, it follows that the distance between either one of the vertices $u, v$ and any vertex of $S_1$ is at most $2$.

We also have $N^+(v)\cap S_1=\emptyset$. Note that $S_1$ cannot have vertices in common with either $N^-(v)$ or $N^-(u)$ otherwise a directed triangle occurs.

The following result is similar to Claim \ref{claimtuv0}.
\begin{claim}\label{claimtuv1}
There exists $m-2$ subsets of $V(D)$, denoted $S_1, S_2, \ldots, S_{m-2}$, such that for every $1\le k\le m-2$ the following properties hold:
\begin{align*}
(a)\, &|S_1|+|S_2|+\cdots |S_k| \ge (1-\alpha)r+(1-\alpha)^2r+\cdots +(1-\alpha)^{k-1} r+(1-\alpha)^k t(u,v)\\
(b)\, &\text{The distance between vertex}\,\, w_1\,\, \text{and any vertex of}\,\, S_k\, \text{is at most}\,\, k.\\
(c)\, &\text{The distance between any of the vertices}\,\, u, v\,\, \text{and any vertex of}\,\, S_k\, \text{is at most}\,\, k+1.\\
(d)\, &\text{The sets}\,\, N^+(v), N^{-}(v), N^{-}(u)\setminus N^{-}(v),\,\, S_1, S_2,\ldots S_{k-1}, \,\,\text{and}\,\, S_k\,\,\text{are mutually disjoint.}
\end{align*}
\end{claim}
\begin{proof}
We prove this claim by finite induction on $k$.
The base case $k=1$ was proved above.

Suppose that for some $2\le k \le m-2$ we have found the sets $S_1, S_2, \ldots S_{k-1}$ with the properties stated in Claim \ref{claimtuv1}.

Let $G_{k}$ be the subdigraph of $D$ induced by $N^+(v) \sqcup S_1 \sqcup S_2\cdots \sqcup S_{k-1}$. Note that since $u\in N^-(v)$, statement $(d)$  implies that $u$ is not a vertex of this subgraph. By the minimality of $D$, there exists a vertex $w_{k} \in N^+(v) \sqcup S_1 \sqcup S_2\cdots \sqcup S_{k-1}$  whose outdegree within $G_k$ is no greater than $\alpha|N^+(v) \sqcup S_1 \sqcup S_2\cdots \sqcup S_{k-1}|= \alpha(|N^+(v)|+|S_1|+|S_2|+\cdots+ |S_{k-1}|)$. On the other hand, $w_k$ is a vertex of $D$, so $d^+_D(w_k) = r = \lceil \alpha n \rceil$.

Let $S_{k} $ be the set of all outneighbors of $w_k$ which are not in $N^+(v) \sqcup S_1 \sqcup S_2\cdots \sqcup S_{k-1}$. Then,
\begin{align*}
&|S_k|\ge r-\alpha(|N^+(v)|+|S_1|+|S_2|+\cdots+ |S_{k-1}|)=r-\alpha r- \alpha(|S_1|+|S_2|+\cdots |S_{k-1}|)\implies\\
&|S_1|+|S_2|+\cdots |S_k|\ge (1-\alpha)r+(1-\alpha)(|S_1|+|S_2|+\cdots +|S_{k-1}|)\implies\\
&|S_1|+|S_2|+\cdots |S_k|\ge (1-\alpha)r+(1-\alpha)^2r+\cdots +(1-\alpha)^{k-1} r+ (1-\alpha)^k t(u,v),
\end{align*}
and this proves the inductive step for statement $(a)$.

Since $w_k\in N^+(v) \sqcup S_1 \sqcup S_2 \sqcup \cdots \sqcup S_{k-1}$, the distance between $w_1$ and $w_k$ is no greater than one of the values $0, 1, 2, \ldots k-1$, depending on whether $w_k\in N^+(v)$, or $w_k\in S_i$, for some $1\le i \le k-1$. Since $S_k$ is a subset of the outneighborhood of $w_k$, the inductive step for part $(b)$ of the claim follows.

Since $(u,w_1), (v,w_1)\in E(D)$, statement $(c)$ follows immediately from $(b)$.

Finally, for proving part $(d)$, it suffices to show that $S_k$ is disjoint from both $N^-(u)$ and $N^-(v)$.
 If $S_k\cap N^-(u)\ne \emptyset$ or $S_k\cap N^-(v)\ne \emptyset$ then, since the distance from either $u$ or $v$ to any vertex of $S_k$ is no greater than $k+1$, one would obtain a directed cycle of length at most $k+2$. However, this contradicts the assumption that $D$ contains no directed cycles of length at most $m$.
 This concludes the proof of Claim \ref{claimtuv1}
\end{proof}
We can now finalize the proof of Case 2.

Since the sets $N^+(v), N^-(v), N^-(u)\setminus N^-(v), S_1, S_2, \ldots, S_{m-2}$ are mutually disjoint
\begin{align*}
&n\ge |N^+(v)| + |N^-(v)|+ |N^-(u)\setminus N^-(v)|+ |S_1|+ |S_2|+ \cdots + |S_{m-2}| \implies\\
&n\ge r+ d^-(v)+q(u,v)+ (1-\alpha)r+(1-\alpha)^2r+\cdots +(1-\alpha)^{m-3}r+(1-\alpha)^{m-2} t(u,v) \implies\\
&n\ge \frac{1-(1-\alpha)^{m-2}}{\alpha}\,r+d^{-}(v)+q(u,v)+(1-\alpha)^{m-2} t(u,v),\quad \text{as stated}.
\end{align*}
The proof of Lemma \ref{mainlemma1} is now complete.
\end{proof}
For proving Theorem \ref{mainthm1} we will simply sum up inequality \eqref{mainineq1} over all edges $(u,v) \in E$.
Denote
\begin{equation}\label{tau}
T:=\sum_{(u,v)\in E}t(u,v),\,\, \text{ the number of transitive triangles in $D$, and}\,\, \tau:=\frac{T}{nr^2}.
\end{equation}
Recall that all vertices of digraph $D$ have outdegree $r=\lceil \alpha n \rceil$. Then
\begin{align*}
&\sum_{(u,v)\in E}n= n^2r,\\
&\sum_{(u,v)\in E}q(u,v)=\sum_{(u,v)\in E}p(u,v)=\sum_{(u,v)\in E}(r-t(u,v))=nr^2-T = nr^2(1-\tau),\,\,\text{and}\\
&\sum_{(u,v)\in E}d^-(v)=\sum_{v\in V}(d^-(v))^2\ge \frac{1}{n}\left(\sum_{v\in V}d^-(v)\right)^2=\frac{1}{n}\left(\sum_{v\in V}d^+(v)\right)^2=nr^2.
\end{align*}
Summing inequalities \eqref{mainineq1} over all edges $(u,v) \in E$ we obtain that
\begin{equation*}
n^2r\ge \frac{1-(1-\alpha)^{m-2}}{\alpha}nr^2+nr^2+(nr^2-T)+(1-\alpha)^{m-2}T,
\end{equation*}
which after dividing both sides by $nr^2$ and rearranging gives
\begin{align}
&\frac{1}{\alpha}\ge \frac{n}{r}\ge \frac{1-(1-\alpha)^{m-2}}{\alpha}+2-\left(1-(1-\alpha)^{m-2}\right)\tau\implies\notag\\
&\tau\left(1-(1-\alpha)^{m-2}\right)\ge 2-\frac{(1-\alpha)^{m-2}}{\alpha}.\label{tauineq1}
\end{align}
Moreover, it is easy to see that the number of transitive triangles is no greater than the number of out-2-claws, that is
\begin{equation}\label{weaktau}
T\le \sum_{v\in V}{d^+(v) \choose 2}=\sum_{v\in V}{r \choose 2}= \frac{n(r^2-r)}{2}< \frac{nr^2}{2},\quad \text{from which it follows that}\,\,\tau <1/2.
\end{equation}
Using this inequality into \eqref{tauineq1} it follows that
\begin{equation*}
(1-\alpha)^{m-2}>\frac{3\alpha}{2-\alpha},
\end{equation*}
contradicting the choice of $\alpha=\alpha(m)$. This proves Theorem \ref{mainthm1}.

As mentioned earlier, Theorem \ref{mainthm1} was proved by Shen \cite{shen1} in the case $m=3$, and by Liang and Xu \cite{LX2, LX5} when $m=4, 5$.
We present some numerical estimates below
\begin{align}\label{estimates}
&ch(3)\le \alpha(3)< 0.35425, ch(4)\le \alpha(4)<0.28866, ch(5)\le \alpha(5)<0.24817\notag\\
&ch(6)\le \alpha(6)< 0.21984, ch(7)\le \alpha(7)<0.19856, ch(8)\le \alpha(8)=0.18182.
\end{align}
While it is easy to prove that $\alpha(m)\longrightarrow 0$ as $m$ approaches infinity, it would be interesting to find its exact order of magnitude. Rewrite the equation that defines $\alpha:=\alpha(m)$ as follows
\begin{align*}
&(1-\alpha)^{m-2}\cdot\left(1-\frac{\alpha}{2}\right)=\frac{3}{2}\alpha \implies \exp\left((m-2)\ln(1-\alpha)+\ln\left(1-\frac{\alpha}{2}\right)\right)=\frac{3}{2}\alpha\implies\\
&\exp\left((m-2)\left(-\alpha-\frac{\alpha^2}{2}-\frac{\alpha^3}{3}-\cdots\right)+
\left(-\frac{\alpha}{2}-\frac{\alpha^2}{8}-\frac{\alpha^3}{24}-\cdots\right)\right)=\frac{3}{2}\alpha\implies\\
&\exp(-(m-2.5)\alpha)\ge \frac{3}{2}\alpha\implies (m-2.5)\alpha\cdot\exp((m-2.5)\alpha)\le \frac{2}{3}(m-2.5) \,\, \text{from which}
\end{align*}
\begin{equation*}
\alpha\le \frac{W_0(\frac{2}{3}(m-2.5))}{m-2.5}.
\end{equation*}
Here, $W_0$ is the real branch of Lambert's omega function.

It follows that $ch(m)\le \alpha(m)=O\left(\frac{\ln{m}}{m}\right)$, and for large values of $m$ this estimate is weaker than the bounds given in \eqref{shenbounds}. However, for small values of $m$, the estimates in \eqref{estimates} are close to the best currently known.


\section{\bf {Slight improvements}}\label{chudnovsky}

In the final portion of the proof of Theorem \ref{mainthm1}, we used that $\tau<1/2$ - see \eqref{weaktau}.

As noticed by Chudnovsky, Seymour and Sullivan \cite{CSS}, this bound is susceptible for improvement. We will present their approach below.

Let $m\ge 3$. A simple digraph $G$ is said to be \emph{$m$-free} if there is no directed cycle of $G$ of length at most $m$. A digraph is \emph{acyclic} if it has no directed cycles. Given an $m$-free digraph, one might ask how many edges must be removed before the graph becomes acyclic.

For a given a digraph $G$, let $\beta(G)$ be the size of the smallest subset $X\subseteq E(G)$ such that $G \setminus X$ is acyclic, and let $\gamma(G)$ be the number of unordered pairs of nonadjacent vertices in $G$, called the \emph{number of missing edges} of $G$.
Chudnovsky, Seymour and Sullivan raised the question of bounding $\beta(G)$ by some function of $\gamma(G)$.

They proved that if $G$ is a $3$-free digraph then $\beta(G)\le \gamma(G)$. This was subsequently improved by Dunkum et al. to $\beta(G)\le 0.88\gamma(G)$, and further by Chen et al. to $\beta(G)\le 0.8616\gamma(G)$. It is conjectured that for a $3$-free digraph, $\beta(G)\le 0.5\gamma(G)$.

Sullivan proved that for an $m$-free digraph $\beta(G)\le \frac{1}{m-2}\gamma(G)$ for $m=4, 5$, and this was generalized by Liang and Xu for all $m\ge 3$.

They also proved that if $G$ is a $4$-free digraph, then $\beta(G)\le \frac{3-\sqrt{5}}{2}\gamma(G)\approx0.3819\gamma(G)$, and that if $G$ is a $5$-free digraph, then $\beta(G)\le (2-\sqrt{3})\gamma(G)\approx 0.2679\gamma(G)$.

Sullivan \cite{sullivanthesis} proposed the following general conjecture which if true would be best possible.
\begin{conj}
If $G$ is an $m$-free digraph with $m\ge 3$ then
\begin{equation*}
\beta(G)\le \frac{2}{(m+1)(m-2)}\gamma(G).
\end{equation*}
\end{conj}
For every $m\ge 3$, let $c_m$ be defined as follows
\begin{equation}\label{C}
c_3=0.8616, c_4=\frac{3-\sqrt{5}}{2},c_5=2-\sqrt{3},\,\,\text{and}\,\, c_m=\frac{1}{m-2}\,\,\, \text{if}\,\, m\ge 6.
\end{equation}
The previous discussion implies the following
\begin{fact}\label{lemma1}
If an $m$-free digraph $G$ has $e$ missing edges, then one can delete from $G$ an additional $c_m e$ edges so that the resulting digraph is acyclic.
\end{fact}
The following lemma generalizes a result of Haxell, Hamburger, and Kostochka.
\begin{lemma}\label{lemma2}
If an $m$-free digraph $G$ has $e$ missing edges, then it has a vertex whose outdegree is no greater than $\sqrt{2c_m e}$.
\end{lemma}
\begin{proof}
Let $d = \lceil\sqrt{2c_m e}\rceil$. By the previous fact, $G$ contains an acyclic digraph $G^\prime$ with at least $|E(G)| - c_m e$ edges. Arrange the vertices of $G^\prime$ in an order $v_1, v_2,\ldots, v_n$ so that there are no backward edges. If $G$ has no vertices with outdegree less than $d$, then for each $ n-d+1\le i \le n$, the set $E(G)\setminus E(G^\prime)$ contains at least $i-(n-d)$ edges starting at vertex $v_i$. Hence,
\begin{equation*}
c_m e\ge |E(G)|-|E(G')|\ge 1+2+\cdots +d=\frac{d^2+d}{2}>\frac{d^2}{2}\ge \frac{(\sqrt{2c_m e})^2}{2}=c_m e,\,\, \text{a contradiction}.
\end{equation*}
\end{proof}
\begin{thm}\label{mainthm2}
Consider the following quantities,
\begin{align}
&\beta(3)=0.35296, \beta(4)=0.28688, \beta(5)=0.24647,\notag\\
&\beta(6)=0.21851, \beta(7)=0.19732, \beta(8)=0.18068\label{betterbounds}.
\end{align}
Then, for every $3\le m\le 8$, any digraph on $n$ vertices with minimum outdegree at least $\beta(m)\cdot n$ contains a directed cycle of length at most $m$, that is, $ch(m)\le \beta(m)$.
\end{thm}
While we prove the above result only for $3\le m\le 8$, similar estimates can be obtained for any given value of $m$. The reason we restrict ourselves to these values of $m$ is three-fold. First, if one compares the bounds in \eqref{estimates} to those in \eqref{betterbounds}, the improvements get progressively smaller. Second, the general statement of Theorem \ref{mainthm2} is rather awkward. Third, for values of $m\ge 14$, Shen's bound \cite{shen3} is stronger than the one we would obtain with the current method.

\section{\bf Proof of Theorem \ref{mainthm2}}

As with Theorem \ref{mainthm1}, we prove Theorem \ref{mainthm2} by induction on $n$. It is easily seen that the statement is valid for $n=3$. Let us now suppose that the theorem  holds for all digraphs with fewer that $n$ vertices, and let $D=(V,E)$ be a counterexample with $n$ vertices. Thus, $D$ is a directed graph with minimum outdegree at least $\lceil \beta(m)\cdot n\rceil$ and $D$ contains no directed cycles of length at most $m$.
To improve readability and maintain consistency set $\alpha:=\beta(m)$. Without loss of generality, we can assume that $D$ is $r$-outregular, where $r= \lceil \alpha n\rceil$, that is, every vertex is of outdegree $r$ in $D$. We intend to reach a contradiction.

\begin{lemma}\label{lemma3}
Let $D$ be a minimal counterexample $r$-outregular, $m$-free digraph. Then, for every $v\in V(D)$
\begin{equation}
d^-(v) \leq \frac{(1 - \alpha)^{m - 1}}{\alpha}r
\end{equation}
\end{lemma}
\begin{proof}
We will construct $m-2$ subsets of $V(D)$, denoted $S_1, S_2,\ldots, S_{m-2}$, with the following properties. For every $1\le k \le m-2$
\begin{align*}
&(a)\,|S_1|+|S_2|+\cdots |S_k|\ge (1-\alpha)r+(1-\alpha)^2r+\cdots +(1-\alpha)^k r.\\
&(b)\, \text{The distance from vertex}\,\, v\,\,\text{to any vertex of}\,\,S_k\,\,\text{is at most}\,\, k+1.\\
&(c)\, \text{The sets}\,\,\{v\}, N^+(v), N^-(v), S_1, S_2, \ldots, S_k\,\,\text{are mutually disjoint}.
\end{align*}
We proceed by induction. Let $G_1$ be the subdigraph induced by $N^+(v)$. By the minimality of $D$, there exists a vertex $w_1\in N^+(v)$ whose outdegree in $G_1$ is no greater that $\alpha|N^+(v)|=\alpha r$. Let $S_1$ be the set of outneighbors of $w_1$ not in $N^+(v)$.
Then clearly, $|S_1|\ge (1-\alpha) r$, and the distance between $v$ and any vertex of $S_1$ is no greater than $2$. Also, note that $S_1$ cannot have vertices in common with any of the sets $\{v\}, N^+(v), N^-(v)$, otherwise a directed cycle of length at most $3$ would occur. This proves the base case $k=1$.

Next, suppose that for some $2\le k \le m-2$ we have found the sets $S_1, S_2, \ldots S_{k-1}$ with the properties stated above

Let $G_{k}$ be the subdigraph of $D$ induced by $N^+(v) \sqcup S_1 \sqcup S_2\cdots \sqcup S_{k-1}$. Note that $v$ is not a vertex of this subgraph. By the minimality of $D$, there exists a vertex $w_{k} \in N^+(v) \sqcup S_1 \sqcup S_2\cdots \sqcup S_{k-1}$  whose outdegree within $G_k$ is no greater than $\alpha(|N^+(v)|+|S_1|+|S_2|+\cdots+ |S_{k-1}|)$. On the other hand, $w_k$ is a vertex of $D$, so $d^+_D(w_k) = r = \lceil \alpha n \rceil$.

Let $S_{k} $ be the set of all outneighbors of $w_k$ which are not in $N^+(v) \sqcup S_1 \sqcup S_2\cdots \sqcup S_{k-1}$. Then,
\begin{align*}
&|S_k|\ge r-\alpha(|N^+(v)|+|S_1|+|S_2|+\cdots+ |S_{k-1}|)=r-\alpha r- \alpha(|S_1|+|S_2|+\cdots |S_{k-1}|)\implies\\
&|S_1|+|S_2|+\cdots |S_k|\ge (1-\alpha)r+(1-\alpha)(|S_1|+|S_2|+\cdots |S_{k-1}|)\implies\\
&|S_1|+|S_2|+\cdots |S_k|\ge (1-\alpha)r+(1-\alpha)^2r+\cdots (1-\alpha)^{k-1} r+ (1-\alpha)^k r
\end{align*}
and this proves the inductive step for statement $(a)$.

Since $w_k\in N^+(v) \sqcup S_1 \sqcup S_2 \sqcup \cdots \sqcup S_{k-1}$, the distance between $v$ and $w_k$ is no greater than one of the values $1, 2, \ldots k$, depending on whether $w_k\in N^+(v)$, or $w_k\in S_i$, for some $1\le i \le k-1$. Since every vertex of  $S_k$ is an outneighbor of $w_k$, the inductive step for part $(b)$ of the claim follows.

Finally, for proving part $(c)$, it suffices to show that $S_k$ is disjoint from both $\{v\}$ and $N^-(v)$.
 If $S_k\cap \{v\}\ne \emptyset$ then one would obtain a directed cycle of length at most $k+1$. If $S_k\cap N^-(v)\ne \emptyset$ then one would obtain a directed cycle of length at most $k+2$. In both cases this contradicts the assumption that $D$ contains no directed cycles of length at most $m$. This completes the proof of the inductive step.

 It follows that
 \begin{align*}
 &n\ge 1+|N^+(v)|+|N^-(v)|+|S_1|+|S_2|+\cdots |S_{m-2}|\implies\\
 &\frac{r}{\alpha}\ge n \ge 1+r+d^-(v)+(1-\alpha)r+(1-\alpha)^2 r+\cdots+ (1-\alpha)^m-2 r \implies\\
 &d^-(v)\le \frac{r}{\alpha}-\frac{1-(1-\alpha)^{m-1}}{\alpha}r-1<\frac{(1-\alpha)^{m-1}}{\alpha} r.
 \end{align*}
\end{proof}
We next introduce some new quantities. For every $3\le m\le 8$ let
\begin{equation}\label{AB}
a_m=\frac{(1-\beta(m))^{m-1}}{\beta(m)}=\frac{(1-\alpha)^{m-1}}{\alpha}, \,\, \text{and}\,\, b_m=\frac{a^2_mc_m+2a_m-1}{2a_m(1+c_m)},
\end{equation}
where the values of $c_m$ are defined in \eqref{C}.

Computing $a_m$ and $b_m$ numerically we obtain
\begin{align*}
&a_3\approx1.18614,\, a_4\approx1.26411,\, a_5\approx1.30809,\, a_6\approx1.33396,\, a_7\approx1.35545,\, a_8\approx1.37055,\\
&b_3\approx0.58522,\,\, b_4\approx0.61209,\, b_5\approx0.62543,\,\, b_6\approx0.63353,\,\, b_7\approx0.63888,\,\, b_8\approx0.64234.
\end{align*}
For each $(u, v)\in E(D)$, let $f(u, v)$ be the number of missing edges in $N^+(u) \cap N^+(v)$. Similarly, for each $v \in V(D)$, let $f(v) = {r\choose 2}- |E(D(N^+(v)))|= $ the number of missing edges in $N^+(v)$, and $t(v) = |E(D(N^+(v)))|=$ the number of transitive triangles in $D$ with source vertex $v$. By definition, $t(v) + f(v) = {r \choose 2}$ for each $v \in V(D)$, and $T = \sum_{v \in V(D)} t(v)=nr^2\tau$, the number of transitive triangles in $D$. It follows that
\begin{equation}\label{f}
\sum_{v\in V(D)} f(v)=n{r\choose 2}- T<nr^2/2- nr^2\tau= nr^2(1/2-\tau).
\end{equation}
\begin{lemma}\label{lemma4}
With the above notations we have
\begin{equation*}
\sum_{(u, v) \in E(D)}f(u, v) < b_m r\sum_{v \in V(D)}f(v).
\end{equation*}
\end{lemma}
\begin{proof}
Let $\overline{E}(D)$ denote the set of missing edges from $D$, that is, the pairs $xy \in {V(D) \choose{2}}$ such that neither $(x, y)$ nor $(y, x)$ is an edge in $D$. Note that
\begin{align*}
\sum_{(u, v) \in V(D)}f(u, v) &= \sum_{xy \in \overline{E}(D)}|E(D(N^-(x) \cap N^-(y)))|\,\,\text {and}\\
\sum_{v \in V(D)}f(v) &= \sum_{xy \in \overline{E}(D)}|N^-(x) \cap N^-(y)|.
\end{align*}
Therefore, the statement of the lemma holds if for every $xy \in \overline{E}(D)$
\begin{equation}
|E(D(N^{-}(x) \cap N^{-}(y))) < Br|N^{-}(x) \cap N^{-}(y)|.
\end{equation}
Let $|N^{-}(x) \cap N^{-}(y)| = q$.

If $q \leq r + 1$ then
$|E(D(N^{-}(x) \cap N^{-}(y))) \leq q(q - 1)/2 < qr/2 < b_m r q$, and we are done since $b_m>1/2$ for all $3\le m\le 8$.

Hence, we can assume that $q \geq r + 2$. Denote by $e$ the number of edges missing from $D(N^{-}(x) \cap N^{-}(y))$. Note that any acyclic digraph on $q$ vertices with maximum outdegree $r$ has at most $\binom{r}{2} + r(q - r) = \binom{q}{2} - \binom{q - r}{2}$ edges.

Since $D(N^-(x)\cap N^-(y))$ itself is $m$-free, Lemma \ref{lemma1} implies that it contains an acyclic subdigraph with at least ${q \choose 2} - (1+c_m)e$ edges. Therefore,

\begin{equation*}
{q\choose 2}-(1+c_m)e\le {q\choose 2}-{q-r \choose 2} \implies e\ge \frac{1}{1+c_m}{q-r \choose 2}.
\end{equation*}
It follows that $E(D(N^-(u)\cap N^-(v))|\le {q\choose 2}-\frac{1}{1+c_m}{q-r \choose 2}$.

Suppose for the sake of contradiction that
\begin{align*}
|E(D(N^{-}(x) \cap N^{-}(y))) \ge b_m r|N^{-}(x) \cap N^{-}(y)|\implies\\
{q\choose 2}-\frac{1}{1+c_m}{q-r \choose 2}\ge b_m r q= \frac{c_m a_m^2+2a_m-1}{2a_m(1+c_m)}qr,
\end{align*}
with after clearing the denominators and simplifying is equivalent to
\begin{equation*}
(q-a_mr)(c_ma_mq+r)-a_m(c_mq+r)\ge 0.
\end{equation*}
However, by Lemma \ref{lemma3} we have $q=|N^-(u)\cap N^-(v)|\le |N^-(v)|=d^-(v)\le \frac{(1-\alpha)^{m-1}}{\alpha}r=a_mr$.
This shows that the above inequality is impossible.
\end{proof}
\begin{lemma}\label{lemma5}
With the notations above we have
\begin{equation*}
\sum_{(u,v)\in E(D)}\sqrt{f(u,v)} <nr^2\sqrt{b_m(1/2-\tau)}.
\end{equation*}
\end{lemma}
\begin{proof}
Using the Cauchy-Schwarz inequality, Lemma \ref{lemma4} and inequality \eqref{f} we obtain
\begin{equation*}
\left(\sum_{(u,v)\in E(D)} \sqrt{f(u,v)}\right)^2\le nr\sum_{(u,v)\in E(D)}f(u,v)\le b_mnr^2\sum_{v\in V(D)} f(v)< b_mn^2r^4(1/2-\tau).
\end{equation*}
This proves Lemma \ref{lemma5}.
\end{proof}
\begin{lemma}\label{mainlemma2}
For every edge $(u,v)\in E(D)$ we have that
\begin{equation}\label{mainineq2}
n\ge \frac{1-(1-\alpha)^{m-2}}{\alpha}\,r+d^{-}(v)+q(u,v)+(1-\alpha)^{m-3}\left(t(u,v)-\sqrt{2c_mf(u,v)}\right).
\end{equation}
\end{lemma}
\begin{proof}
Note the slight difference between the statement above and that of Lemma \ref{mainlemma1}.
If $t(u,v)=0$ then there is nothing to prove.
If $t(u,v)>0$, Lemma \ref{lemma2} implies the existence of a vertex $w_1\in N^+(u)\cap N^+(v)$ which has at most $\sqrt{2c_mf(u,v)}$ outneighbors in the subdigraph induced by $N^+(u)\cap N^+(v)$. It follows that $w_1$ has at most $\sqrt{2c_mf(u,v)}+p(u,v)$ outneighbors in $N^+(v)$.

Let $S_1$ be the set of outneighbors of $w_1$ which do not belong to $N^+(v)$. Then
\begin{equation*}
|S_1|\ge r-(\sqrt{2c_mf(u,v)}+p(u,v))=t(u,v)-\sqrt{2c_mf(u,v)}.
\end{equation*}
We recursively construct the sets $S_2, S_3, \ldots, S_{m-2}$, with the following properties: for every $1\le k\le m-2$
\begin{align*}
(a)\, &|S_1|+|S_2|+\cdots |S_k| \ge (1-\alpha)r+(1-\alpha)^2r+\cdots (1-\alpha)^{k-1} r+(1-\alpha)^{k-1}|S_1|\\
(b)\, &\text{The distance between vertex}\,\, w_1\,\, \text{and any vertex of}\,\, S_k\, \text{is at most}\,\, k.\\
(c)\, &\text{The distance between any of the vertices}\,\, u, v\,\, \text{and any vertex of}\,\, S_k\, \text{is at most}\,\, k+1.\\
(d)\, &\text{The sets}\,\, N^+(v), N^{-}(v), N^{-}(u)\setminus N^{-}(v),\,\, S_1, S_2,\ldots S_{k-1}, \,\,\text{and}\,\, S_k\,\,\text{are mutually disjoint.}
\end{align*}
The proof is almost identical to that of Claim \ref{claimtuv1}, as the only difference is the estimate for the size of $S_1$.
\end{proof}
Summing inequalities \eqref{mainineq2} over all edges $(u,v)\in E(D)$, we obtain a similar inequality to the one in Theorem \ref{mainthm1}:
\begin{equation*}
n^2r\ge \frac{1-(1-\alpha)^{m-2}}{\alpha}nr^2+nr^2+nr^2(1-\tau)+(1-\alpha)^{m-3}T-(1-\alpha)^{m-3}\sqrt{2c_m}\sum_{(u,v)\in E(D)}\sqrt{f(u,v)},
\end{equation*}
which after using Lemma \ref{lemma5} and dividing by $nr^2$ gives
\begin{equation*}
\frac{1}{\alpha}\ge \frac{n}{r}\ge \frac{1-(1-\alpha)^{m-2}}{\alpha}+2-\tau+(1-\alpha)^{m-3}-(1-\alpha)^{m-3}\sqrt{b_mc_m(1-2\tau)}.
\end{equation*}
Rearranging, we obtain
\begin{equation}
\tau(1-(1-\alpha)^{m-3})+(1-\alpha)^{m-3}\sqrt{b_mc_m(1-2\tau)}\ge 2-\frac{(1-\alpha)^{m-2}}{\alpha}.
\end{equation}\label{tauineq2}
This inequality complements the earlier inequality \eqref{tauineq1}
\begin{equation}\label{tauineq1bis}
\tau\left(1-(1-\alpha)^{m-2}\right)\ge 2-\frac{(1-\alpha)^{m-2}}{\alpha}.
\end{equation}
Recall that $\tau<1/2$. If $\tau$ is close to $1/2$ then inequality \eqref{tauineq2} is the stronger one. This is the reason why
the result in Theorem \ref{mainthm1} can be slightly improved.

For the choices of $\alpha=\beta(m)$ given in Theorem \ref{mainthm2}, inequality \eqref{tauineq1bis} implies that
$\tau>\tau_m^{*}$ where $\tau_3^*=0.4726, \tau_4^*=0.4625, \tau_5^*=0.4615, \tau_6^*=0.4673, \tau_7^*=0.4669$, and $\tau_8^*=0.4688$.

However, it is straightforward (albeit tedious) to check that inequality \eqref{tauineq2} is not satisfied by any $\tau \in (\tau_m^*, 1/2)$.
We thus reached the desired contradiction. The proof of Theorem \ref{mainthm2} is now complete.

\end{document}